%% file: pseudoauts0104.tex
\documentclass[12pt]{amsart}
\usepackage{amssymb}
\usepackage{amsmath} 
\usepackage{amsthm} 
\usepackage{calc} 
\usepackage{verbatim} 
\input{shortcuts}

\def\sqr#1#2{{\vcenter{\hrule height.#2pt              %qed
     \hbox{\vrule width.#2pt height#1pt\kern#1pt
     \vrule width.#2pt}
     \hrule height.#2pt}}}
\def\square{\mathchoice\sqr{5.5}4\sqr{5.0}4\sqr{4.8}3\sqr{4.8}3}
\def\qed{\hskip4pt plus1fill\ $\square$\par\medbreak}

\newcommand{\ve}{\mathbf{e}}

\renewcommand{\div}{\mathop{\rm{Div}}}
\renewcommand{\supp}{\mathop{supp}}
\newcommand{\aut}{\mathrm{Aut\,}}

\title{Pseudoautomorphisms with invariant elliptic curves}

\author{Eric Bedford}
\address{Department of mathematics\\ Indiana University\\ Bloomington IN 47405,  { \em Current address: } Stony Brook University, Stony Brook, NY, 11794 }
\email{bedford@indiana.edu}
\author {Jeffery Diller}
\address{Department of Mathematics\\ University of Notre Dame\\Notre Dame IN  46556}
\email{diller.1@nd.edu}
\author{Kyounghee Kim}
\address{Department of mathematics\\Florida State University\\Tallahassee FL 32306}
\email{kim@math.fsu.edu}

\begin{document}
\maketitle
%\markboth{\today}{\today}

\bigskip
Nontrivial automorphisms of complex compact manifolds are typically rare and more typically non-existent.  It is interesting to understand which manifolds admit automorphisms, how plentiful they are on any given manifold, and what further special properties distinguish a particular automorphism, or family of automorphisms.  These problems have enjoyed much attention in the past fifteen years, motivated largely by work in complex dynamics (e.g. Cantat's thesis \cite{cantat1999dynamique}).  In this introduction, we give a quick account of some of this research, introducing in particular the more general category of pseudoautomorphisms, which occur more frequently in higher dimensions than automorphisms.  The final aim of our paper is to present a concrete alternative approach to some recent existence results \cite[Theorems 1.1 and 3.1]{perroni2011pseudo} of Perroni and Zhang for pseudoautomorphisms with invariant elliptic curves on rational complex manifolds.  Our methods lead to explicit formulas which are especially simple (see Theorems \ref{T:L} and \ref{multiprojL}) when the pseudoautomorphisms correspond to the `Coxeter element' in an infinite, finitely generated reflection group.

The topological entropy of an automorphism is a non-negative number that measures the complexity of point orbits.  `Positive entropy'  will serve as a precise and reasonable necessary condition for a map to be dynamically interesting.  In complex dimension one, i.e.\ on closed Riemann surfaces, there are no automorphisms of positive entropy.  In dimension two, Cantat \cite{cantat1999dynamique} showed that only three types of complex surfaces can carry automorphisms of positive entropy:  tori, $K3$ surfaces (or certain quotients), or rational surfaces.  Automorphisms of tori are essentially linear.  The cases of $K3$ and rational surfaces are much more interesting.  Dynamics of automorphisms of $K3$ surfaces were studied in detail by Cantat \cite{cantat2001dynamique}.  McMullen \cite{mcmullen2002dynamics} constructed examples which exhibit rotation domains (two dimensional `Siegel disks').   The family of all $K3$ surfaces has dimension  20, and the maximum dimension of a continuous family of $K3$ surface automorphisms is even smaller.   By contrast, there are continuous families of rational surface automorphisms which have arbitrarily large dimension \cite{Bedford-Kim:2010}.   

It is known \cite{Nagata, Dolgachev:2008} that rational complex surfaces $X$ that carry automorphisms of positive entropy are in fact modifications (i.e.\  compositions of point blowups) $\pi:X\to\cp^2$ of the complex projective plane $\cp^2$.  Thus a rational surface automorphism $F_X:X\to X$ with positive entropy descends via $\pi$ to a birational `map' $F:\cp^2\dasharrow\cp^2$ which is locally biholomorphic at generic points but also has a finite union of exceptional curves that are contracted to points and conversely a finite collection $I(F)$ of indeterminate points which are (in a precise sense) each mapped to an algebraic curve.  Since the group of all birational maps $F:\cp^2\dasharrow \cp^2$ is quite large, this suggests trying to find automorphisms by looking at a promising family of plane birational maps and identifying those elements whose exceptional/indeterminate behavior can be eliminated by repeated blowup.  

The papers \cite{Bedford-Kim:2006} and \cite{Bedford-Kim:2009} pursued exactly this idea for a well-chosen two parameter family of quadratic birational maps $F:\cp^2\dasharrow\cp^2$ and found a countable set of parameters for which there exists a modification $\pi:X\to \cp^2$ lifting $F$ to an automorphism $F_X:X\to X$ with positive entropy.  A generic \ quadratic birational map $F$ on $\cp^2$ has three exceptional lines $\Sigma_j$, $j=0,1,2$ and three points of indeterminacy $e_0,e_1,e_2$.  The maps $F$ considered in \cite{Bedford-Kim:2006} and \cite{Bedford-Kim:2009} all share the further property that $F(\Sigma_0) = e_1$ and $F(\Sigma_1) = e_2$.  If the parameter is chosen correctly, one can further arrange that $F^n(\Sigma_2) = e_0$ for some (minimal) $n>0$.  In this case, the map $F$ lifts to an automorphism $F_X:X\to X$ of the rational surface $\pi:X\to \cp^2$ obtained by blowing up $e_1$, $e_2$ and the points $F^j(\sigma_2)$, $1\leq j\leq n$.  In effect, the exceptional curves and indeterminate points cancel either other out in the blown up space $X$.  It was observed in \cite{Bedford-Kim:2006} that $F_X$ has finite order when $n\leq 6$, zero entropy when $n=7$ and positive entropy for all $n\geq 8$.  In fact, it is generally true that one must blow up at least ten points in $\cp^2$ to arrive at a rational surface that admits an automorphism with positive entropy.

There is a curious dichotomy, discussed at length in \cite{Bedford-Kim:2009}, among the automorphisms discovered in \cite{Bedford-Kim:2006}.  For any fixed $n\geq 8$, there are finitely many maps in the family satisfying $F^n(\Sigma_2) = e_0$.  Some, but not all, of these have the additional feature that they preserve a cubic curve $C\subset\cp^2$ with a cusp singularity.  This curve $C$, when it exists, contains all points blown up by the modification $\pi:X\to\cp^2$, and the proper transform of $C$ by $\pi$ is a rational elliptic curve preserve by the automorphism $F_X$.  

By \emph{requiring} the existence of an invariant elliptic curve, McMullen \cite{mcmullen2007dynamics} showed how one can arrive at the examples in \cite{Bedford-Kim:2009} synthetically.  His approach is to begin with a plausible candidate $F_X^*:\pic(X)\to\pic(X)$ for the induced action of the automorphism on the picard group of the surface and then seek an a surface $X$ and an automorphism $F_X:X\to X$ that `realizes' $F_X^*$.  Here $\pic(X)$ is equivalent to $\Z^{1+N}$, where $N$ is the number of blowups needed to create $X$, and if the identification between $\pic(X)$ and $\Z^{1+N}$ is chosen appropriately, the intersection form on $\pic(X)$ becomes the standard Lorentz metric on $\Z^{1+N}$.  The natural candidates for the action $F_X^*$ are isometries in a certain coxeter group acting on $\Z^{1+N}$.  If one seeks to realize $F_X^*$ with an automorphism $F_X$ that fixes an elliptic curve $C$, then the action $F_X^*$ on $\pic(X)$ must restrict to a corresponding action $(F_X|_C)^*$ on $\pic(C)$.  It is well-known that the identity component of $\pic(C)$ identifies naturally with the regular part of $C$ (see e.g. the appendix to \cite{Diller:2009}).  Using this identification and some theory of Coxeter groups, McMullen gave a sufficient condition for realization of $F_X^*$ by an automorphism.  In particular, the maps with invariant elliptic curves discovered in \cite{Bedford-Kim:2009} turn out to be realizations of the so-called Coxeter element in the isometry group of $\Z^{1+N}$.

The ideas in \cite{Bedford-Kim:2009} and \cite{mcmullen2007dynamics} were combined in later work.  In particular, \cite{Diller:2009} described all possible rational surface automorphisms with invariant elliptic curves that are obtained as lifts of \emph{quadratic} birational maps $F:\cp^2\dasharrow\cp^2$.  Uehara \cite{Uehara:2010} showed that whether or not one can actually realize a plausible (in McMullen's sense) candidate action $F_X^*$, one can always construct a rational surface automorphism $F_X$ that is closely related to $F_X^*$ in the sense that the topological entropy of $F_X$ has the correct value (the log of the spectral radius of $F_X^*$).

Constructing automorphisms on rational $k$-folds seems to be much more difficult when $k\geq 3$.  At present, the only known examples with positive entropy appear in \cite{Oguiso:2013ty}.  If one works only with rational $k$-folds obtained as finite compositions $\pi:X\to \cp^k$ of point blowups over projective space, then a recent result of Truong \cite{Truong:2012uq} and Bayraktar-Cantat \cite{Bayraktar-Cantat} says us that any automorphism of $X$ must have zero entropy.  So with this constraint on the manifold $X$, one must settle for constructing maps which are not quite automorphisms.   A birational map $F_X:X\dasharrow X$ is a \emph{pseudoautomorphism} \cite{Dolgachev-Ortland} if there are sets $S_1,S_2\subset X$ of codimension $\ge2$ such that $F:X-S_1\to X-S_2$ is biregular.  Equivalently, the image of a hypersurface under both $F_X$ and $F_X^{-1}$ is always a hypersurface and never a subvariety of codimension larger than one. 

Having expanded the class of maps we seek, we also modify our criterion for determining which maps are dynamically interesting.  Entropy is not an invariant of birational conjugacy (see Guedj \cite{Guedj:Panorama}), so we employ a related but birationally invariant number, the (first) dynamical degree $\delta(F_X)$.  For a pseudoautomorphism, $\delta(F_X)$ is just the spectral radius of the induced action $F_X^*:\pic(X)\to\pic(X)$.  When, as in dimension two, $F_X$ is a genuine automorphism, celebrated results of Gromov \cite{Gromov} and Yomdin \cite{Yomdin} imply that $\log\delta(F_X)$ is the entropy of $F_X$.  In fact this equality holds generically \cite{MR2752759} for birational maps of $\cp^k$, but it is not known precisely when it fails.  At any rate, the first dynamical degree is much easier to work with for pseudoautomorphisms, so it seems reasonable to substitute $\delta(F_X)>1$ for positive entropy in our criterion for dynamically interesting maps.

% 
% \footnote{We are defining the first dynamical degree for the iterates of maps of blowups of $\cp^k$.  We say that this is a birational invariant, so $\delta(F)=\delta(F_X)$.  Should we also say that in the case of pseudo-automorhisms (and regular maps in general) this is also the growth rate of the iterates of $f^*$ acting on $Pic(X)$ and $H^{1,1}(X)$?  My point is that there are several definitions that coincide in the case of blowups of $\cp^k$.  Should we mention this a little more completely?} The \emph{degree} $\deg F_X$ of $F_X$ of a birational map $F_X:X\to X$ is the quantity $F_X^{-1}(H)\cdot L$ where $H$ and $L$ and are the pullbacks to $X$ of a generic hyperplane $H\subset \cp^k$ and line $L\subset \cp^k$.   If $X=\cp^2$ (i.e. $F_X = F$), this is just the common degree of the homogeneous component functions defining the birational map $F$ underlying $F_X$.  The \emph{(first) dynamical degree} of $F_X$ is then the asymptotic exponential growth rate of $\deg F_X^n$
% $$
% \delta(F_X):=\lim_{n\to\infty} \left ({\rm deg}(F_X^n)\right)^{1/n},
% $$
% If $F_X$ is an everywhere regular map\footnote{we just pointed out that $F_X$ is never everywhere regular}, then celebrated results of Gromov \cite{Gromov} and Yomdin \cite{Yomdin} imply that $\delta(F_X)>1$ is equivalent to positive entropy.  

Following McMullen's approach, Perroni and Zhang \cite{perroni2011pseudo} recently showed that one can also construct pseudoautomorphisms $F_X:X\dasharrow X$ with $\delta(F_X) > 1$ on point blowups $X$ of $\cp^k$ (and more generally, on point blowups of products $\cp^k\times\dots\times\cp^k$).  As in McMullen, they begin with a candidate for the pullback action $F_X^*:\pic(\cp^k)\to\pic(\cp^k)$, chosen from a certain reflection group.  They proceed by requiring $F_X$ to preserve an \emph{elliptic normal curve} $C$ (discussed here in \S\ref{S:ENcurve}) and exploiting the group structure that $C_{reg}$ inherits from its identification with $\pic_0(C)$; ant then they obtain a sufficient criterion for realizing the proposed action $F_X^*$ with a pseudoautomorphism.  Their criterion implies in particular that when $F_X^*$ is the `coxeter element' in the reflection group, then $F_X^*$ is realizable.  One has $\delta(F_X^*) > 1$ in this case, so the resulting pseudoautomorphism is, by our standard, dynamically interesting.  

Our goal in this article is to follow ideas from \cite{Diller:2009} in order to make the construction of Perroni and Zhang more explicit, arriving at precise and fairly simple formulas for the maps they discovered.  The approach is as follows.  We begin with what we call \emph{basic cremona maps} $F := S\circ J\circ T^{-1}$ on $\cp^k$ (discussed at length in \S\ref{S:basic cremona maps}).  
Here $S,T\in PGL(k+1,{\bf C})$ are linear automorphisms, and $J:\cp^k\dasharrow\cp^k$ is the Cremona involution $J[x_0:\cdots:x_k]=[x_0^{-1}:\cdots:x_k^{-1}]$.  The exceptional hypersurfaces of $J$ (and hence $F$) are the coordinate hyperplanes $\Sigma_j:=\{x_j=0\}$.  The image $J(\Sigma_j)$ is the point $e_j=[0:\cdots:0:1:0:\cdots:0]$ obtained by intersecting all the other coordinate hyperplanes.  
The exceptional hypersurfaces of $J$ (and hence $F$) are the coordinate hyperplanes $\Sigma_j:=\{x_j=0\}$.  The image $J(\Sigma_j)$ is the point $e_j=[0:\cdots:0:1:0:\cdots:0]$ obtained by intersecting all the other coordinate hyperplanes.  Conversely, $J$ is indeterminate along the codimension two set consisting of points where two or more coordinate hyperplanes meet.  

The effect of the linear maps $S$ and $T$ is to vary the locations of the exceptional hypersurfaces and their images for $F$ and $F^{-1}$.  Specifically, the columns $S(\ve_j)$ of $S$ are images of exceptional hypersurfaces for $F$ and the columns of $T$ are the images of exceptional hypersurfaces of $F^{-1}$.  One can use this freedom to try and arrange that there exist integers $n_j$ and a permutation $\sigma$ of $\{0,\dots,n\}$ such that 
$$
F^{n_j-1}(S(\ve_j))=T(e_{\sigma(j)}), \eqno(*)
$$
where none of the intermediate points $F^n(\ve_j)$, $1\leq n<n_j-1$, lie in $I(F)$.  Under these conditions it is straightforward to see that when one blows up the intermediate points, then $F$ lifts to pseudoautomorphism $F_X:X\dasharrow X$.  The data $\{(n_0,\dots, n_k),\sigma\}$ was called \emph{orbit data} in \cite{Bedford-Kim:2004}, where it was shown that the orbit data alone are sufficient to determine the dynamical degree $\delta(F_X)$.   For given orbit data, the condition $(*)$ amounts to a polynomial system of equations satisfied by the entries of the matrices $S$ and $T$.  The simple appearance of this condition is deceptive, however, because it involves equations of many variables and polynomials of very high degree.  Moreover, $S$ and $T$ are taken from the noncompact group $\aut(\cp^k)$, so one cannot reliably apply intersection theory even to guarantee existence of solutions.

Things become simpler if we require that $F$ preserves the elliptic normal curve $C\subset\cp^k$.  The main result of \S\ref{S:basic cremona maps}, and the first step in our construction of pseudoautomorphisms, is Theorem \ref{good br maps}.  It gives an explicit description of those basic cremona maps that fix the elliptic normal curve $C\subset\cp^k$ in terms of the points $T(\ve_j), S(\ve_j)$ and the (affine) restriction $F|_C$ of $F$ to $C$.  In particular, preserving $C$ is essentially equivalent to requiring that $C_{reg}$ contains the points $S(\ve_j), T(\ve_j)$ and therefore also, all intermediate points in $(*)$.  The orbits $F^{n_j-1}(S(\ve_j))$ are now obtained by iterating an affine map inside a one dimensional set, so the equations imposed by the orbit data are much more tractable.

In sections \S\ref{S:pseudo} and \S\ref{S:Pk_coxeter}, we therefore use Theorem \ref{good br maps} to derive a formula for $F$ in the case where the orbit data corresponds to the action $F_X^*:\pic(X)\to\pic(X)$ of the Coxeter element.  It turns out that the formulas are much simpler after a linear conjugation, letting $F = L\circ J$, where $L = T^{-1}\circ S$.  Lemma \ref{lem:tj} and Theorem \ref{T:L} combine to give the entries of the matrix $L$ and hence a formula for $F$.

The connection between rational surface automorphisms and Coxeter groups is more straightforward in dimension two because in that case the intersection product on a surface gives a quadratic form on $\pic(X)$.  The pullback action  of any automorphism of $X$ is then an isometry of $\pic(X)$ that decomposes into a sequence of geometrically natural reflections.  In higher dimensions, one needs an auxiliary identification between intersection product of divisors and an actual quadratic form.  We describe this identification in Section~\ref{S:coxeter}.  If $\pi:X\to\cp^k$ is the blowup of $N$ points $p_1,\dots,p_N\in\cp^k$, then the cohomology group $H^{2}(X;{\bf Z})$ is naturally isomorphic to $Pic(X)$.  A basis for either of these groups is given by the (pullback of the) class of a general hyperplane $E_0\subset\cp^k$, together with the exceptional blowup divisors $E_j$ over $p_j$.  It turns out that there is a unique element $\Phi\in H^{2k-4}(X;{\bf Z})$  such that the inner product
$$
\langle D,D'\rangle:= H\cdot H'\cdot \Phi
$$
on classes $D,D'\in\pic(X)$ is invariant by any  pseudoautomorphism $F_X:X\dasharrow X$ corresponding to a basic cremona map.  Further, we can decompose the action $F_X^*:\pic(X)\to\pic(X)$ into simple reflections as in the surface case.  Finally, the maps we arrive at here can be seen to represent the `Coxeter element', which is the isometry of $\pic(X)$ obtained by composing all of the basic generating reflections.

For simplicity we have confined our attention to pseudoautomorphisms on modifications of $\cp^k$ with invariant cuspidal elliptic curves, but our methods work more generally.  In particular, as in Perroni-Zhang, we can replace $\cp^k$ with a product of projective spaces $\cp^k\times\dots\times\cp^k$.  In section \ref{S:multi} we sketch the main details of our method for \emph{biprojective} spaces $\cp^k\times\cp^k$.  It is worth noting that our methods work when the elliptic normal curve, is replaced by various other elliptic curves.    For instance, in section \ref{S:multi} we also say a few words about replacing the elliptic normal curve with $k+1$ concurrent lines.

%
%
%we describe two other families of curves in ${\bf P}^k$ for which the construction given here works without essential change.
%
%Let us finally note that the elliptic normal curve is only one possibility for a distinguished invariant curve.  Though for brevity's sake we do not carry them out here, our methods work when the elliptic normal curve is replaced by other elliptic curves (e.g. $k+1$ concurrent lines in $\cp^k$).  But it is not clear to us what constitutes the full range of possibilities for an invariant curve.  For instance, not all curves of degree $k+1$ (the degree of the elliptic normal curve) in $\cp^k$ are elliptic.  While these might be plausible candidates for an invariant curve, these some to us much more difficult to work with.

We do not know whether our methods can be used to produce pseudoautomorphisms with non-elliptic invariant curves $C\subset{\bf P}^k$.  For instance, a union of $k+1$ lines in $\cp^k$ has the same degree as the elliptic normal curve, but when the lines are mutually disjoint, the union is far from elliptic and seems much harder to work with.  On the other hand, a related construction of a pseudoautomorphism in ${\bf P}^3$ is given in \cite{Bedford-Cantat-K}; this construction involves iterated blowups along an invariant curve quite different from the curves treated here.

%However, unlike the surface case, such curves are not always elliptic and it certainly seems to us much more difficult to work with non-elliptic degree $k+1$ curves.  \footnote{I would like to have a small remark somewhere where we comment more specifically on the possibilities for different invariant curves.}

\section{From elliptic normal curves...}\label{S:ENcurve}

By design, the pseudoautomorphisms we construct in this paper will all have a distinguished invariant curve.  In this section we describe the curve and its key properties.

Let $[x_0,\dots,x_k]$ be homogeneous coordinates on $\cp^k$.  An irreducible complex curve $C\subset\cp^k$ is \emph{rational} if there is a holomorphic parametrization $\psi:\cp^1\to C$. Using affine coordinates on the domain and homogeneous coordinates on the range of $\psi$, one can write $\psi(t) = [\psi_0(t):\dots:\psi_k(t)]$ where $\psi_j(t)$ are polynomials with no common factor.  The \emph{degree} of $C$ is the number of intersections, counted with multiplicity, between $C$ and any hyperplane $H\subset\cp^k$ that does not contain $C$.  This is a topological invariant, independent of the parametrization $\psi$.  Nevertheless, one sees readily that $\deg C = \max_j\deg\psi_j$.

Let $C = \gamma(\cp^1)\subset \cp^k$ be the complex curve of degree $k+1$ given by $\gamma(t) = [1:t:\dots:t^{k-1}:t^{k+1}]$ for all $t\in\C$ and $\gamma(\infty) = [0:\dots:0:1]$.  In analogy with rational normal curves, $C$ is sometimes called the \emph{elliptic normal curve}.  The following results further justify this term.  

\begin{prop}
\label{encurve} The curve $C$ has a unique smooth inflection point at $\gamma(0)$, and a unique singularity at $\gamma(\infty)$, which is an ordinary cusp.  Moreover,
\begin{itemize}
 \item No hyperplane $H\subset\cp^k$ contains $C$.  
 \item No proper linear subspace $L\subset \cp^k$ contains more than $\dim L+1$ points of $C_{reg}$, counted with multiplicity.
 \item Any other degree $k+1$ curve $C'$ that is not contained in a hyperplane and that has a cusp singularity is equal to $T(C)$ for some $T\in\aut(\cp^k)$.
\end{itemize}
 \end{prop}

Note that in the second item, if $\dim L < k-1$ then the multiplicity of $L\cap C$ at $p$ is defined to be the minimal (i.e. generic) multplicity of $H\cap C$ at $p$ among hyperplanes $H$ containing $L$.  

\begin{proof}
The initial assertions follow from elementary computations.  That $C$ is not contained in a hyperplane follows from linear independence of the monomials $\{1,t,\dots,t^{k-1},t^{k+1}\}$.  The second item now follows for hyperplanes $L=H\subset\cp^k$ from the fact that $H\cdot C = k+1$.

Suppose instead that $L\subset \cp^k$ is a linear subspace with codimension at least two.  Let $S\subset C-L$ be a set of $k-\dim L-2$ regular points of $C$.  Then there exists a hyperplane $H$ that contains $L$, $S$ and the cusp of $C$.  Since $H$ does not contain $C$, and the cusp is a point of multiplicity (at least) two in $H\cap C$, we infer that
$$
\# L\cap C_{reg} + \# S + 2 \leq k+1 = H\cdot C,
$$
which implies the second item in the proposition.  

It remains to establish the third item.
If $C'$ is another curve with degree $k+1$ and $p\in C'$ is a singular point, then we can choose a hyperplane $H\ni p$ that meets $C'$ at $k$ distinct points $p_1=p, p_2,\dots, p_k$.  Thus
$H\cdot C' \geq (k-1) + \mu$, where $\mu$ is the multiplicity $p$ in $C'$.  Necessarily then $\mu = 2$, and we have equality.  In particular, no other point of $C'$ is singular.  

To see that $C'$ is isomorphic to $C$ via $\aut(\cp^k)$, we show inductively that there exists a flag $L_0 := \{p\} \subset L_1 \subset \dots \subset L_{k-1}\subset\cp^k$ of linear subspaces such that $\dim L_j = j$, and $L_j\cap C = \{p\}$ with multiplicity $j+2$.   To this end, suppose we have a partial flag $L_0\subset\dots \subset L_{j-1}$ as described.  The set of $j$ dimensional subspaces containing $L_{j-1}$ is naturally parametrized by $\cp^{k-j+1}$ (i.e. each is determined by $L_{j-1}$ and a choice of normal vector to $L_{j-1}$ at $p$).  Thus we have a map $L:C-\{p\} \to \cp^{k-j+1}$ given by $q\mapsto L(q)$ where $L(q)$ is the unique $j$ dimensional subspace containing $q$ and $L_{j-1}$.  Since $L$ is meromorphic on $C$, and $C$ is one dimensional, the map $L$ extends holomorphically across $p$.  For all $q\in C_{reg}$, we have that $L(q)$ contains $p$ with multiplicity at least $j+1$ and $q$ with multiplicity at least $1$.  Since $L(q)\to L(p)$ as $q\to p$, it follows that $L_j := L(p)$ contains $p$ with multiplicity $\mu$ at least $j+2$.
Now if $S\subset C$ is a generic set of $k-j-1$ points, then $S$ and $L_j$ span a hyperplane $H$, and 
$$
k+1 = H\cdot C \geq \#S + \mu = k-j-1+\mu.
$$
Hence $\mu = j+2$ exactly as claimed.  

Finally, we let $T\in\aut(\cp^k)$ be a linear transformation satisfying $T(L_j) = \{x_0 = \dots x_{k-j-1} = 0\}$.  Let $\psi = (\psi_0,\dots,\psi_k):\cp^1\to C'$ be a polynomial parametrization satisfying $\psi(0) = [0,\dots,0, 1] = T(p)$.  By hypothesis, $\deg \psi_j \leq k+1$ for each $j$.   Since $L_j$ meets $C'$ at $p$ to order $j+2$ and $p$ is a cusp, we have $\psi_{k-j-1}(t) = t^{j+2}\tilde\psi_{k-j-1}(t)$ for some polynomial $\tilde\psi_{k-j-1}$ such that $\tilde \psi_{k-j-1}(0)\neq 0$.  Thus we can apply a further `triangular' transformation $S\in\aut(\cp^k)$ so that $S\circ \psi(t) = [t^{k+1},t^k,\dots,t^2,1]$.  Thus $t^{k+1}S\circ \psi(1/t)=\gamma(t)$;  i.e. $S\circ T(C') = C$.
\end{proof}

For any $p\in C_{reg}$, we write $[p]\in\div(C)$ to indicate the divisor of degree $1$ supported at $p$.  The following classical fact about $C$ will be essential in what follows.

\begin{prop}[Group Law]
\label{grouplaw}
The set $\pic_0(C)$ of linear equivalence classes of degree zero divisors on $C$ is isomorphic (as an algebraic group, i.e. as a Riemann surface and a group) to $\C$, with isomorphism given by $\mu=\sum n_j[\gamma(t_j)] \mapsto \sum n_j t_j$ whenever $\gamma(\infty) \notin \supp\mu$.  If, moreover, $\mu = D|_C$ is the restriction to $C$ of a divisor $D\subset \div\cp^k$, then $\sum n_j = \deg D$ and $\sum n_j t_j  =0$.
\end{prop}

In particular, any divisor $\delta\in\div(C)$ of degree zero is equivalent to $[\gamma(t)] - [\gamma(0)]$ for some $t\in\C$ and addition in $\pic_0(C)$ is given by $\sum ([\gamma(t_j)] - [\gamma(0)]) = [\gamma(\sum t_j)] - [\gamma(0)]$.

\begin{proof}
Equivalence between $\pic_0(C)$ and $\C = C_{reg}$ is classical.%\footnote{What is the difference between `(truly) classical' and more simply `classical'? It was a private joke to myself that I forgot to expunge.  In French math, classical seems to refer to anything that happened more than a year or so ago, but to me the word means something more like `known to some of my distant ancestors.'}  
The point here is that a divisor $\delta\in\div(C)$ is principle if and only if 
$\delta = \div h$ for some rational $h:C\to\cp^1$ satisfying (without loss of generality) $h(\gamma(\infty)) = 1$.  But because $\gamma'(\infty)$ vanishes to first order, we see that $h\circ \gamma:\cp^1\to\cp^1$ is a meromorphic function satisfying $h\circ \gamma(\infty) = 1$ and $(h\circ\gamma)'(\infty) = 0$.  Writing $h\circ \gamma = P/Q$, one sees that this is equivalent to $\deg P = \deg Q$ and $\sum_{P(t)=0} t = \sum_{Q(t)=0} t$, where the roots are included with multiplicity in each sum.

Let $H\subset \cp^2$ be the hyperplane defined by $x_k$.  Note that $H|_D = (k+1)[\gamma(0)]$.  And if $D\in\div(\cp^k)$ is another effective divisor with degree $d$ such that $\gamma(0)\notin\supp D$, then $D|_C = \sum_{j=1}^{(k+1)d} \gamma(t_j)$ for some $t_j\in C$.  Since $D - dH$ is principal in $\cp^k$, we have that $[\gamma(\sum t_j)] - [0] \sim (D-dH)|_C \sim 0$ in $\pic_0(C)$.  So $\sum t_j = 0$.  
\end{proof}

\begin{prop}
\label{linearstab}
Let $T\in\aut(\cp^k)$ be a linear transformation satisfying $T(C) = C$.  Then $T = T_\lambda$ for some $\lambda\in\C^*$, where $T_\lambda:[x_0,\dots,x_k]\mapsto [x_0,\lambda x_1,\dots \lambda^{k-1}x_{k-1},\lambda^{k+1}x_k]$
\end{prop}

\begin{proof}
That $T_\lambda(C) = C$ is easily checked.  On the other hand, if $T(C) = C$ for some $T\in\aut(C)$, then the facts that $T$ is a biholomorphism mapping lines to lines and that $C$ has a single cusp $\gamma(\infty)$ and a single inflection point $\gamma(0)$ imply that $T(\gamma(0)) = \gamma(0)$ and $T(\gamma(\infty)) = \gamma(\infty)$.  So $T|_{C_{reg}}$ corresponds via $\gamma$ to the linear map $t\mapsto \lambda t$ for some $\lambda$.  That is, $T|_{C}=T_\lambda|_C$.  Since $C$ is not contained in a hyperplane, we infer $T = T_\lambda$.
\end{proof}

\section{...to basic cremona maps...}
\label{S:basic cremona maps}

If $X\to \cp^k$ is a rational surface obtained by blowing up subvarieties of $\cp^k$, and $F_X:X\to X$ is a pseudoautomorphism, then $F_X$ descends to a birational map $F:\cp^k\dasharrow\cp^k$ on projective space.  Our method for constructing pseudoautomorphisms reverses this observation.  That is, we begin with an appropriate family of birational maps, and then we identify elements $F$ of the family that lift by blowup to pseudoautomorphisms $F_X$.  In this section we describe the family of birational maps we will use.

Let $J:\cp^k\dasharrow\cp^k$, given by $J[x_0,\dots,x_k]=x_0\dots x_k\cdot[1/x_0:\dots:1/x_k]$, be the standard cremona involution of degree $k$ on $\cp^k$.  Since $J$ is a monomial map, the algebraic torus $(\C^*)^k$ is totally invariant by $J$, and $J$ restricts to a biholomorphism (and group isomorphism) on this set.  The complement of $(\C^*)^k$ is the union of all coordinate hyperplanes $\{x_j=0\}$, and each of these is exceptional, contracted by $J$ to a point.  The indeterminacy set of $J$ is the union of all linear subspaces obtained by intersecting two or more coordinate hyperplanes.  For any non-empty set of indices $I\subset\{0,\dots,k\}$, we have $J(\bigcap_{i\in I} \{x_i=0\}) = \bigcap_{i\notin I} \{x_i=0\}$.

Given $S,T\in\aut(\cp^k)$, we will refer to $F := S\circ J\circ T^{-1}$ as a \emph{basic cremona map}.  The exceptional set of $F$ consists of the images $T(\{x_j=0\})$ of the coordinate hyperplanes, these are mapped by $F$ to the points $S(\ve_j)$.  The same is true, with $T$ and $S$ reversed, for $F^{-1}$.
Note that if $\Lambda\in\aut(\cp^k)$ is any `diagonal' transformation, preserving all coordinate hyperplanes, then $\Lambda\circ J = J\circ\Lambda^{-1}$.
Hence replacing $S$ and $T$ by $S\circ\Lambda$ and $T\circ \Lambda$ does not change $F$.  Otherwise, $S$ and $T$ are determined by $F$.

Note that for \emph{any} birational transformation $F:X\to Y$ one has a natural pullback (or pushforward) map $F^*:\div(Y)\to \div(Y)$ on divisors, and this descends to at map $F^*:\pic(X)\to\pic(Y)$ on linear equivalence classes. On the other hand, we will employ the convention for any hypersurface $H\subset Y$ that $F^{-1}(H) = \overline{F^{-1}(H-I(F^{-1}))}$ is the set-theoretic \emph{proper} transform of $H$.  One has that $F^{-1}(H)$ is irreducible
%\footnote{ (or empty) --- can't be empty?? Jeff: the empty set is irreducible} 
when $H$ is irreducible and that, as a reduced divisor, $F^{-1}(H) = F^*H$ precisely when no hypersurface contracted by $F$ has its image contained in $H$.

\begin{prop}
\label{hyperplane image}
Suppose that $F = S\circ J\circ T^{-1}$ is a basic cremona map. Let $H\subset\cp^k$ be a hyperplane that does not contain any of the points $T(\ve_j)$.  Then $F^{-1}(H) = F^*H$ is an irreducible degree $k$ hypersurface, and the multiplicity of $F^*H$ at any point $p\in I(F)\cap F^{-1}(H)$ is $\ell + 1$ where $\ell$ is the number of exceptional hyperplanes $T(\{x_j = 0\})$ containing $p$.
\end{prop}

\begin{proof}
This may be computed directly.
\end{proof}

We will say that $F$ is \emph{centered on $C$} if all points $T(\ve_j)$ lie in $C_{reg}$.

\begin{prop}
\label{exceptionalpts}
Suppose that $F = S\circ J\circ T^{-1}$ is a basic cremona transformation centered on $C$.  
\begin{itemize}
 \item The cusp $\gamma(\infty)$ of $C$ lies outside the exceptional set of $F$
 \item For each $j$, there exists at most one point $p\notin \{T(\ve_j):0\leq j\leq k\}$ where $C$ meets the exceptional hyperplane $T(\{x_j=0\})$. If $p$ exists, it is not contained in any other exceptional hyperplane; and the intersection between $C$ and $T(\{x_j=0\})$ is transverse at $p$. 
 \item If $p$ does not exist, then $C$ is tangent to $T(\{x_j=0\})$ at some point $T(\ve_i)$, $i\neq j$.  In this case $C$ meets all other exceptional hyperplanes transversely at $T(\ve_i)$. 
 \end{itemize}
 In particular $C\cap I(F) = \{T(\ve_j):0\leq j \leq k\}$.
\end{prop}

\begin{proof}
The curve $C$ intersects the exceptional hyperplane $T(\{x_j= 0\})$ at $k+1$ points counting multiplicity.  By hypothesis these include the $k$ points $T(\ve_i)$ for all $i\neq j$, and none of these are the cusp of $C$.  Consequently, either $C$ is tangent to $T(\{x_j=0\})$ at one of the points $T(\ve_i)$, or $C$ meets $T(\{x_j=0\})$ transversely at exactly one other point $p$.  

In particular, if $p$ exists, then it cannot be the cusp of $C$.  Moreover, $p$ cannot lie in another exceptional hyperplane $T(\{x_\ell=0\})$, $\ell\neq j$, because then the $k-2$-dimensional subspace $T(\{x_\ell=x_j=0\})$ would contain $k$ points: $p$ and $T(\ve_m)$ for all $m\neq \ell,j$, contradicting Proposition \ref{encurve}.

The same argument shows that when $p$ does not exist and $C$ is tangent to $T(\{x_j=0\})$ at $T(\ve_i)$, then $C$ cannot be tangent to any other exceptional hyperplane at $T(\ve_i)$.
\end{proof}

Proposition \ref{exceptionalpts} gives us information about the image of $C$ under $F$.

\begin{cor}
\label{centered}
Let $F$ be as in the previous proposition and let $C' = F(C)$.  Then 
\begin{itemize}
 \item $F$ maps the cusp of $C$ to a cusp of $C'$; and 
 \item $S(\ve_j)\in C'$ for all $j$.
\end{itemize}
In particular $C'$ is not contained in any hyperplane.
\end{cor}

\begin{proof}
The first conclusion follows from the fact that $F$ is regular near the cusp $\gamma(\infty)$.

If $C$ meets an exceptional hyperplane $T(\{x_j=0\})$ at a point $p$ not contained in another coordinate hyperplane, then $F(p) = S(J(\{x_j=0\})) = S(\ve_j) \in C'$.  Similarly, one computes readily that if $C$ is tangent to $\{x_j=0\}$ at $\ve_i$, then $(F\circ\gamma)(t) = S(\ve_j) \in C'$, where $\gamma(t) = T(\ve_j)$. 

The final assertion follows from independence of the points $S(\ve_j)$.
\end{proof}

\begin{prop}
\label{rightdegree}
Let $F$ and $C'$ be as in Corollary \ref{centered}. Then $\deg C' = k+1$. 
\end{prop}

\begin{proof}
We have $\deg C'\geq k$, because otherwise the $k+1$ components of $J\circ T^{-1}\circ\gamma(t)$ would all be polynomials of degree smaller than $k$ and therefore dependent. That is, $C'$ would lie in a hyperplane, contrary to the previous propositon.  Moreover, because $C'$ has a cusp, $C'$ is not a rational normal curve.  Therefore $\deg C' \geq k+1$.  

For the reverse inequality, let $H\subset \cp^k$ be a hyperplane that meets $C'$ transversely at a set $K$ of distinct regular points outside the exceptional hyperplanes for $F^{-1}$.  Then $\deg C' = \# K = \#F^{-1}(K)$.  Moreover, all points $F^{-1}(K)$ are regular for $C$, and none lies in an exceptional hyperplane of $F$. Thus we may use Proposition \ref{hyperplane image} to compute
$$
F^*H|_{T^{-1}(C)} \geq \sum_{p\in F^{-1}(K)} [p] + (k-1)\sum [T(\ve_j)],
$$
and then infer 
$$
\deg C' = \# F^{-1}(K) \leq (k+1)\deg F^*H - (k-1)(k+1) = k+1.
$$
\end{proof}

From Propositions \ref{encurve} and \ref{rightdegree} and Corollary \ref{centered}, we immediately obtain

\begin{cor}
\label{centeredimage}
Let $F$ be a basic cremona map centered on $C$.  Then $F(C)$ is an elliptic normal curve and $F^{-1}$ is centered on $F(C)$.
\end{cor}

Let us say that a basic cremona map $F$ \emph{properly fixes $C$} if $F$ is centered on $C$ and $F(C) = C$.  Note that then $F$ induces an automorphism on $C$ which corresponds via $\gamma$ to an affine transformation $t\mapsto\delta t + \tau$ for some $\delta\in\C^*$ and $\tau \in \C$.   We now arrive at the main result of this section.   A consequence of the group law (Proposition \ref{grouplaw}), it gives us a good family of birational maps to work with when looking for pseudoautomorphisms.

\begin{thm}
\label{good br maps}
Suppose $\delta\in\C^*$ and $t^+_j\in\C$, $0\leq j\leq k$, are distinct parameters satisfying $\sum t_j^+ \neq 0$.  Then there exists a unique basic cremona map $F=S\circ J\circ T^{-1}$ and $\tau\in\C$ such that
\begin{itemize}
 \item $F$ properly fixes $C$ with $F|_C$ given by $F(\gamma(t)) = \gamma(\delta t + \tau)$. 
 \item $\gamma(t_j^+) = T(\ve_j)$ for each $0\leq j\leq k$. 
\end{itemize}
Specifically, 
\begin{itemize}
 \item $\tau = \frac{k-1}{k+1}\delta\sum t_j^+$; and 
 \item $S(\ve_j) = \gamma\left(\delta t_j^+ - \frac{2\tau}{k-1}\right)$.
\end{itemize}
\end{thm}

Note that the points $T(\ve_j)$ and $S(\ve_j)$ \emph{almost} determine $T$ and $S$ (and therefore $F$).  Below, it will be convenient to use invariance of the cusp to eliminate the remaining ambiguity.

\begin{proof} 
For existence of $F$, we infer from the condition $\sum t_j^+\neq 0$ and Proposition \ref{grouplaw} that the points $\gamma(t_j^+)$ are independent in $\cp^k$.  Therefore, there exists
$T\in\aut(\cp^k)$ such that $T(\ve_j) = \gamma(t_j^+)$.  Then $J\circ T$ is a basic cremona map centered on $C$, and $J\circ T(C)$ is an elliptic normal curve.  So by Proposition \ref{encurve} there exists $S\in\aut(\cp^k)$ such that $F = S\circ J\circ T(C) = C$.  The restriction $F|_C$ is given by $F(\gamma(t)) = \gamma(\alpha t + \tau)$ for some $\alpha\in\C^*$ and $\tau\in C$.  Replacing $S$ with $T_\lambda\circ S$, $\lambda := \delta/\alpha$, we may assume that $\alpha = \delta$.  Thus $F$ is the basic cremona map we seek.

For uniqueness and the remaining assertions about $F$, suppose we are given $S, T\in\aut(\cp^k)$ such that $F = S\circ J\circ T^{-1}$ satisfies the given conditions.  If $H\subset \cp^k$ is a generic hyperplane, then $H$ meets $C$ in $k+1$ distinct points $\{\gamma(t_0),\dots,\gamma(t_k)\}$ such that $\sum t_j = 0$, and none of these points lies in the exceptional set of $F$.  Hence $F^{-1}(\gamma(t_j)) = \gamma((t_j-\tau)/\delta)$.  Also, $F^*H = F^{-1}(H)$ is a hypersurface of degree $k$ that contains with multiplicity $k-1$ all points $T(\ve_j) = \gamma(t_j^+)$ that are images of exceptional hyperplanes for $F^{-1}(H)$.  This accounts for all $k(k+1)$ points of intersection between $C$ and $F^{-1}(H)$.  Using Proposition \ref{grouplaw} we convert this to the following relationship among parameters
$$
0 = \sum \frac{t_j - \tau}{\delta} + k\sum t_j^+ = -\frac{(k+1)\tau}{\delta} + (k-1)\sum t_j^+. 
$$
That is, $\tau = \frac{k-1}{k+1}\delta\sum t_j^+$.

Now consider an exceptional hyperplane $H = T(\{x_j=0\})$ for $F$.  We have $H\cap C = \{\gamma(t_i^+):i\neq j\} \cup \{p_j\}$, where $p_j = \gamma(s_j)$ is as in Proposition \ref{exceptionalpts}.  Thus on the one hand, $F(p_j) = S(\ve_j) = \gamma(\delta s_j + \tau)$, and on the other hand Proposition \ref{grouplaw} gives us that $s_j = -\sum_{i\neq j} t_i^+ = t_j^+ - \frac{k+1}{k-1}\frac\tau\delta$.  So we arrive at
$$
S(\ve_j) = \gamma\left(\delta t_j^+ - \frac{2\tau}{k-1}\right).
$$

To see that the above information completely determines $F$, let $p\in\cp^k$ be a generic point and $H$ be the hyperplane through $p$ and $k-1$ points $\gamma(t_j^+)$. Then $F(H)$ is a hyperplane containing the corresponding set of $k-1$ points $p_j = S(\ve_j)$.  In addition $H\cap C$ contains two more points $\gamma(s_1), \gamma(s_2) \in C$, so $F(H)$ contains the images $\gamma(\delta s_1 + \tau), \gamma(\delta s_2+\tau)$.  The points $p_j$ and $\gamma(\delta s_j+\tau)$ more than suffice to determine $F(H)$.  

By varying the set of $k-1$ points $\gamma(t_j^+)$ used to determine $H$, we can find a collection of $k$ hyperplanes that intersect uniquely at $p$ and whose images are also hyperplanes intersecting uniquely at $F(p)$.  Since the image hyperplanes are completely determined by the data $t_j^+$, $\delta$, we see that $F(p)$ is uniquely determined by the same data. 
\end{proof}

\section{...to pseudoautomorphisms...}\label{S:pseudo}
A birational map $F:X\to Y$ is a \emph{pseudoautomorphism} if neither $F$ nor $F^{-1}$ contracts hypersurfaces.  Equivalently, $F^{\pm 1}$ have trivial critical divisors.
When combined with the following additional observation, Theorem \ref{good br maps} allows us to create many pseudoautomorphisms.

\begin{prop}
\label{pseudoaut}
Let $F:\cp^k\dasharrow\cp^k$ be a basic cremona map properly fixing $C$.  Suppose there is a permuation $\sigma:\{0,1,\dots,k\}\self$ and, for each point $S(\ve_j)\in I(F^{-1})$, an integer $n_j\geq 1$ such that 
\begin{itemize}
 \item $F^j(S(\ve_j))\notin I(F)$ for $0\leq j < n_j-1$, and
 \item $F^{n_j-1}(S(\ve_j))= T(\ve_{\sigma j})\in I(F)$.  
\end{itemize}
Then $F$ lifts to a pseudoautomorphism $F_X:X\dasharrow X$ on the complex manifold obtained by blowing up the points $S(\ve_j),\dots,F^{n_j-1}(S(\ve_j))$ for all $0\leq j\leq k$.  Moreover, $F_X$ is biholomorphic on a neighborhood of the proper transform $C_X$ of $C$.
\end{prop}

\begin{proof}
We can write $F = \pi_+\circ \tilde F \circ \pi_-^{-1}$, where $\pi_-:\Gamma_-\to \cp^k$ is the blowup of all points $T(\ve_j)\in I(F^{-1})\cap C$, $\pi_+:\Gamma\to\cp^k$ is the blowup of all points $T(\ve_j)\in I(F)\cap C$, and $\tilde F:\Gamma_-\to\Gamma_+$ is a pseudoautomorphism.  More precisely, $I(\tilde F)$ consists of proper transforms $\pi_-^{-1}S(\{x_i=x_j=0\})$, $i\neq j$, of intersections between distinct exceptional hyperplanes; $I(\tilde F^{-1})$ similarly consists of lifts $\pi_+^{-1}T(\{x_i = x_j=0\}$ of intersections between $F^{-1}$ exceptional hyperplanes; and $\tilde F$ restricts to a biholomorphism
$$
\Gamma_- - I(\tilde F) \to \Gamma_+ - I(\tilde F^{-1}).
$$
Proposition \ref{exceptionalpts} therefore tells us that $I(\tilde F)\cap \pi_-^{-1}(C) = \emptyset = I(\tilde F^{-1})$ so that $\tilde F$ maps a neighborhood of $\pi_-^{-1}(C)$ biholomorphically onto a neighborhood of $\pi_+^{-1}(C)$.

Now if $\sigma:X_1\to X_0 := \cp^k$ is the blowup of the image $S(\ve_j)$ of some exceptional hyperplane $T(\{x_j=0\})$, and $F_1:X_1\dasharrow X_1$ is the lift of $F_0:= F$ to $X_1$, then the proper transform of $T(\{x_j=0\})$ is no longer exceptional for $F_1$.  On the other hand the exceptional hypersurface $E = \sigma^{-1}(S(\ve_j))$ is either contracted by $F_1$ or, if $S(\ve_j) \in I(F)$, mapped by $F_1$ onto the proper transform of an exceptional hyperplane $S(\{x_{\sigma j} = 0\})$.  In any case, the map $F_1$ again admits a decomposition $F_1 = \pi_+\circ \tilde F_1\circ \pi_-^{-1}$ where $\pi_\pm$ are compositions of point blowups centered at distinct points in $\sigma^{-1}(C)$, and $\tilde F_1$ is a pseudoautomorphism mapping a neighborhood of $(\sigma\circ\pi_-)^{-1}(C)$ onto a neighborhood of $(\sigma\circ \pi_+)^{-1}(C)$.  

We may therefore proceed inductively blowing up the points $S(\ve_j), F(S(\ve_j)), \dots, F^{n_j-1}(S(\ve_j))$ until the map $F$ lifts to a new birational map with one less exceptional hypersurface than $F$.  Moving on to another point $S(\ve_j)$ and repeating, the the hypothesis of this proposition allow us to finally lift $F$ to a birational map $F_X:X\dasharrow X$ with no exceptional hypersurfaces.  Clearly in the end, $F_X(C_X)= C_X$ and $F_X$ is biholomorphic near $C_X$.
\end{proof}

We have a natural restriction map $tr:\div(X)\to \div(C_X) \cong \div(C)$, obtained by intersecting divisors on $X$ with $C_X$ (and then pushing forward by $\pi|_{C_X}:C_X\to C$, which is an isomorphism).  The map preserves linear equivalence and so descends to a quotient map $tr:\pic(X)\to\pic(C_X)\cong\pic(C)$.  
Moreover, the final conclusion of Proposition \ref{pseudoaut} guarantees that pullback commutes with restriction:  i.e. $(F|_C)^*\circ tr = tr \circ F_X^*$.  

\begin{cor}  
\label{multiplier}
Let $F$, $X$ be as in Proposition \ref{pseudoaut}.  Then the pullback operator $(F|_C)^*$ acts as follows.
\begin{itemize}
\item $tr(K_X) \mapsto tr(K_X)$;
\item $\mu \mapsto \delta\mu$ for all $\mu\in\pic_0(C)$, where $\delta\in\C^*$ is a root of the characteristic polynomial of $F_X^*$.
\end{itemize}
In particular, when $K_X\cdot C_X \neq 0$, this information completely characterizes $(F|_C)$.
\end{cor}
\begin{proof}
The first assertion follows from $F_X^* K_X = K_X$, which holds because $F_X$ is a pseudoautomorphism.

For the second assertion, note that $(F|_C)^*$ restricts to the group automorphism on $\pic_0(C) \cong \C$ given by multiplication by the constant $\delta\in\C^*$ from Theorem \ref{good br maps}.  To see that $P(\delta) = 0$, where $P$ is the characteristic polynomial of $F_X^*$, let $C^\perp\subset\pic(X)$ denote the subspace represented by divisors $D\in\div(X)$ such that $D\cdot C_X = 0$; i.e. $tr(D)$ represents an element of $\pic_0(C)$.  Then $tr(F_X^*D) = \delta \, tr(D)$.  Hence, since $P$ has integer coefficients, we may infer that $P(\delta)\,tr(D) = tr(P(F_X^*)D) = 0$.  We infer that $P(\delta) = 0$ as long as $tr(C^\perp)$ is non-trivial.  But $tr(C^\perp)$ includes all classes of the form $(k+1)[p] - (k+1)[\gamma(0)] = tr((k+1)E - H)$, where $p\in I(F)$ is a point of indeterminacy, $E = \pi^{-1}(p)$ is the corresponding exceptional hypersurface, and $H$ is a hyperplane section.  Since $I(F)$ consists of $k+1$ distinct points, $k>0$ of these classes are non-trivial in $\pic_0(C)$.  

The final assertion follows from the facts that $\pic(C)$ is generated by $\pic_0(C)$ together with one (any) other class of degree one, and that $\deg tr(K_X) = K_X\cdot C_X$.  
\end{proof}

\begin{rem}  We have $K_X\cdot C_X\neq 0$ in all the cases we consider.  If $N$ is the number of blowups comprising $\pi$, then one computes that $K_X \cdot C_X = N(k-1)-(k+1)^2$.  Hence $K_X\cdot C_X$ vanishes only when $k=2, N=9$ or $k=3, N=8$ or $k=5, N=9$.  The last of these does not occur since $N\geq k+1 = \# I(F)$.
\end{rem}

\begin{cor}
\label{no translations}
Let $F$, $X$ be as in Proposition \ref{pseudoaut}.  If $K_X\cdot C_X\neq 0$ and $F|_C$ is a translation, then $F$ is linearly conjugate to the standard cremona map $J$.
\end{cor}

\begin{proof}
The condition that $F|_C$ is a translation is equivalent to the condition that $(F|_C)^*$ is the identity operator on $\pic_0(C)$, i.e. $\delta=1$ in Corollary \ref{multiplier}. 
Since $tr(K_X)$ is also fixed by $(F|_C)^*$ it follows that $(F|_C)^*$ is the identity operator on all of $\pic(C)$.  Therefore $F|_C = \id$.  

We infer that $n=n(p) = 1$ for all $p\in I(F^{-1})$ and therefore $F^2$ is an automorphism of $\cp^k$, mapping each exceptional hyperplane for $F$ to an exceptional hyperplane for $F^{-1}$.  Since $F^2(C) = C$, Proposition \ref{linearstab} further implies that $F^2 = \id$.  That is, $F$ is linearly conjugate to $J$.
\end{proof}

\section{...to formulas}\label{S:Pk_coxeter}

In this section, we derive a formula for a basic cremona map $F:\cp^k\dasharrow\cp^k$ that fixes an elliptic normal curve $C$ and lifts via point blowups along $C$ to a pseudoautomorphism $F_X:X\dasharrow X$ as in Proposition \ref{pseudoaut}.  Specifically, in order to arrive at a formula, we begin with $n>0$ and suppose that $F= S \circ J \circ T^{-1}$ properly fixes $C$ and satisfies the hypotheses of Proposition \ref{pseudoaut} as follows.
\begin{equation}\label{E:orbitdata}
\begin{aligned}
&S(\ve_j) = T(\ve_{j+1})\quad \text{ for }  j=0,1, \dots, k-1, \\
&F^{n-1}(S(\ve_k)) =T(\ve_0), \\
 &F^j(S(\ve_k))  \not\in I(F) \ \ \text{ for } 0 \le j < n-1.\\
\end{aligned}
\end{equation}
That is, $F$ has `orbit data' $(n_0,\dots, n_{k-1},n_k) = (1,\dots, 1, n)$ with cyclic permutation $\sigma:j\mapsto j+1\mod k$. 

We assume that $C= \{ \gamma(t) = [1,t,\dots, t^{k-1}, t^{k+1}], t \in \C\cup \{\infty\} \}$ is in standard form and that $F|_C$ is not a translation.  Note that by Corollary \ref{no translations}, the case of a translation can not lead to $\delta>1$.  Thus we may conjugate $F$ by a linear map $T_\lambda$ to arrange that $\gamma(1)$ is the unique fixed point of $F|_C$ different from the cusp $\gamma(\infty)$, i.e. 
\[
F|_C \ :\ \gamma(t) \mapsto \gamma(\delta (t-1) + 1), \quad \text{and} \quad F|_C^{-1}\ :\ \gamma(t) \mapsto \gamma \left(\frac{1}{\delta} (t-1) +1\right).
\]
Let us also suppose that $\delta \in \C$ is not a root of unity.  As in Theorem \ref{good br maps}, we let $t_j^+$ denote the parameters for points of indeterminacy of $F$: 
\[\gamma(t_j^+) = T(e_j), \ j=0,1,\dots, k.\]

\begin{lem}\label{lem:tj}
The multiplier $\delta$ is a root of the polynomial
\begin{equation*}
 \label{E:salempoly}
(\delta^{n+k}-1)(\delta^2-1) - \delta(\delta^{k+1} - 1)(\delta^{n-1}-1),
\end{equation*}
and the parameters $t_j^+$ are given by
$$
t_j^+ =  \delta^j\frac{k+1}{k-1}\cdot \frac{\delta^2 - 1}{\delta(\delta^{k+1}-1)} - \frac{2}{k-1}. 
$$
\end{lem}

It is convenient that the orbit length $n$ enters into the formula for $t_j^+$ only through the polynomial defining $\delta$.

\begin{proof}
Using the formula for $S(\ve_j)$ from Theorem \ref{good br maps} and $\tau = 1-\delta$, one rewrites the first condition in \eqref{E:orbitdata} as
$$
t_j^+ + \frac{2}{k-1} = \delta \left(t_{j-1}^+ + \frac{2}{k-1}\right) = \dots = \delta^j\left(t_0^+ + \frac{2}{k-1}\right)
$$
for all $1\leq j\leq k$.  Similarly, one rewrites the third condition in \eqref{E:orbitdata} as
$$
\delta^n\left(t_k^+ + \frac{2}{k-1}\right) = t_0^+ + \frac{2}{k-1} + (\delta^{n-1}-1)\frac{k+1}{k-1}.
$$
Comparing this second equation with the case $j=k$ in the first gives
$$
t_0^+ + \frac{2}{k-1} = \frac{\delta^{n-1}-1}{\delta^{n+k}-1}\cdot \frac{k+1}{k-1}.
$$
Hence
$$
\sum_{j=0}^k \left(t_j^+ + \frac{2}{k-1}\right) = \left(\sum_{j=0}^k \delta^j\right)\left(t_0^+ + \frac{2}{k-1}\right)
= \frac{\delta^{k+1}-1}{\delta - 1}\cdot \frac{\delta^{n-1}-1}{\delta^{n+k}-1}\cdot \frac{k+1}{k-1}
$$
On the other hand, the formula for $\tau$ from Theorem \ref{good br maps} gives an alternative expression
$$
\sum_{j=0}^k \left(t_j^+ + \frac{2}{k-1}\right) = \frac{k+1}{k-1} \left(\frac{1}{\delta} +1\right).
$$
Comparing the last two equations, we arrive at
$$
\frac{\delta^{n-1}-1}{\delta^{n+k}-1} = \frac{\delta^2 - 1}{\delta(\delta^{k+1}-1)}.
$$
This gives us the defining polynomial for $\delta$ and allows us to revise our formula for $t_0^+$ to the desired equation
$$
t_0^+ + \frac{2}{k-1} = \delta^j\frac{k+1}{k-1}\cdot \frac{\delta^2 - 1}{\delta(\delta^{k+1}-1)}.
$$
The formulas for $t_j^+$, $j>0$ follow immediately.
\end{proof}

Let us consider a set of indexes  \[\Gamma = \{(2,4),(2,5),(2,6),(2,7), (3,4),  (3,5),(4,4),(5,4)\} \cup \{ (k,n): n \le 3, k\ge 2\}\]

\begin{cor}\label{C:salem}
If $(k,n) \in \Gamma$ then the polynomial for $\delta$ given in (\ref{E:salempoly}) is a product of cyclotomic factors. If $(k,n) \not\in \Gamma$, the polynomial (\ref{E:salempoly}) is the product of cyclotomic factors and a Salem polynomial.
\end{cor}

\begin{proof}
After factoring out $(\delta-1)$ the equation of $\delta$ given in (\ref{E:salempoly})  can be written as
\[ 
\begin{aligned}
\chi(k,n)\ :=\ &\delta^n(\delta^{k+2}- \delta^{k+1}-\delta^k +1) + \delta^{k+2}-\delta^2-\delta+1\\
=\ &\delta^{k+2}(\delta^{n}- \delta^{n-1}-\delta^{n-2} +1) + \delta^{n}-\delta^2-\delta+1\\
\end{aligned}
\] 
Thus we see that if the largest real root is bigger than one then it is strictly increasing to a Pisot number as  $n \to \infty$ and the same is true as $k\to \infty$. The above equation is known as a characteristic polynomial for the coxeter element of a reflection group $W(2, k+1, n-1)$ with $T$ shaped Dynkin diagram. (See section 5 for the connection between $F$ and coxeter elements.) The characteristic polynomial for the coxeter element of such a reflection group is the product of cyclotomic factors and Salem polynomials. (See \cite{McMullen:2002} Proposition 7.1) For all $k\ge 2$, if $n=1$ then $F$ is equivalent to a standard cremona involution $J$. For $n=2, 3 $ we have
\[ \chi(k,2) = (\delta-1) (\delta^{k+2}-1), \qquad \chi(k,3)= (\delta-1)^2 (\delta+1) (\delta^{k+1}+1) \]
If $(k,n)=(2,8), (3,6),(4,5),(5,5)(6,5)$, then $\chi(k,n)$ has a root bigger than one. We get this corollary by checking directly for the pairs in $\Gamma$.
\end{proof}

Thus applying Corollary \ref{multiplier}, we have
\begin{cor}
A basic cremona map $F$ on $\cp^k$ satisfying \eqref{E:orbitdata} has dynamical degree $>1$ if and only if $(n,k) \not \in \Gamma$ and the multiplier for $F|_C$ is not a root of unity.
%If $(k,n) = (2,7), (3,5)$ or $(5,4)$ then the degree of $F^n$ grows quadratically.
\end{cor}

\begin{cor}
\label{noverlaps}
Let $\delta, t_j^+$ be as in Corollary \ref{lem:tj}.  Then all points $T(\ve_j)$, $1\leq j\leq k$ and $F^{-i}(T(\ve_0))$, $0\leq i \leq n-1$ in \eqref{E:orbitdata} are distinct. 
\end{cor}

\begin{proof}
It suffices to show that the parameters corresponding to these points are distinct.

From Lemma \ref{lem:tj} we have for $0\le i< j\le k$,
$$
t_j^+ - t_i^+ = (\delta^j-\delta^i) \cdot\frac{k+1}{k-1}\cdot \frac{\delta^2 - 1}{\delta(\delta^{k+1}-1)}. 
$$
Since $\delta$ is not a root of unity and $k$ is a positive integer, it follows that $t_j^+ \ne t_i^+$ for $i \ne j$. 

Furthermore, if $0\leq j\leq k$ and $1\leq i\leq n-1$ are indices such that $t_j^+ = \delta^{-i}(t_0^+ - 1) + 1$, then Lemma \ref{lem:tj} tells us that
\[ \delta (\delta^{k+1} -1) (\delta^i-1) - (\delta^{i+j}-1)(\delta^2-1)=0\]
Notice that we have 
\[ 
\begin{aligned}
 \delta (\delta^{k+1} -1) (\delta^i-1) 
&\ =\  (\delta-1)^2 \delta (\delta^k+ \delta^{k-1} + \cdots +1) (\delta^{i-1} + \delta^{i-2} +\cdots +1) ,\\
&\ =\  (\delta-1)^2 ( \delta^{k+i}+  \sum_{s=2}^{k+j-1}  c_s \delta^s  + \delta) ,\\
\end{aligned}
\]
where $c_s = \min\{ s, k+1,i+1\} \ge 2$. We also have  
\[
\begin{aligned}
(\delta^{i+j}-1)(\delta^2-1)
&\ =\ (\delta-1)^2 (\delta^{i+j-1} + \delta^{i+j-2} + \cdots + 1) (\delta+1)\\
&\ =\ (\delta-1)^2 (\delta^{i+j} +2 \sum_{s=1}^{i+j-1} \delta^s +1)\\
\end{aligned}
\]
Now $\delta $ is a Galois conjuage of the largest real root $\delta_{r}$ of the polynomial given in Lemma \ref{lem:tj}. It follows that the above equation should be divided by a minimal polynomial of $\delta_{r}$. In other word, $\delta_{r}>1$ must satisfy the above equation. 
We have three different cases:

\begin{itemize}
\item{case 1:} if $j<k$ then by comparing the terms we have 
\[ \delta (\delta^{k+1} -1) (\delta^i-1) - (\delta^{i+j}-1)(\delta^2-1)>0 \text{\ \ for\ \ } \delta>1 \]
\item{case 2:} if $j=k\ge 2$ and $i\ge 2$, then some of $c_s \ge 3$ and thus
\[ \delta (\delta^{k+1} -1) (\delta^i-1) - (\delta^{i+j}-1)(\delta^2-1)>0 \text{\ \ for\ \ } \delta>1 \]
\item{case 3:} if $j=k$ and $i=1$ then 
\[
\begin{aligned}
\delta & (\delta^{k+1} -1) (\delta^i-1) - (\delta^{i+j}-1)(\delta^2-1)\\
&\ =\  (\delta^{k+1} -1) (1-\delta) <0 \text{\ \ for\ \ } \delta>1 
\end{aligned}\]
\end{itemize}
Thus we have the second part of this Corollary.
\end{proof}

The logic of the computations in Lemma \ref{lem:tj} is essentially reversible.  Hence it follows from that lemma and Corollary \ref{noverlaps} that

\begin{thm}
\label{better br map}
Let $\delta$ be a root of $(\delta^{n+k}-1)(\delta^2-1) - \delta(\delta^{k+1}-1)(\delta^{n-1} - 1)$ that is not also a root of unity.  Then
there exists a basic Cremona transformation $F:\cp^k\dasharrow \cp^k$ centered on $C$ with multiplier of $F|_C$ equal to $\delta$ such that $F$ satisfies the orbit data conditions \eqref{E:orbitdata}.  Up to linear conjugacy, $F = S \circ J \circ T^{-1}$ is uniquely specified by the further conditions
$$
T(\ve_j) = \gamma\left(\delta^j\frac{k+1}{k-1}\cdot \frac{\delta^2 - 1}{\delta(\delta^{k+1}-1)} - \frac{2}{k-1}\right).
$$
\end{thm}

The matrices $S$ and $T$ defining the basic cremona map $F = S\circ J\circ T^{-1}$ in this theorem must satisfy 
\[
\begin{aligned}
T\ =\ & \left[\begin{matrix} a_0 \gamma(t_0^+) & a_1\gamma(t_1^+) &  \cdots   & & & \cdots & a_k\gamma(t_k^+) \end{matrix} \right],\\
S\ =\ & \left[\begin{matrix} b_0 \gamma(t_1^+) & b_1\gamma(t_2^+) &  \cdots  & b_{k-1}\gamma(t_k^+) & b_k \gamma((t_0^+-1)/\delta^{n-1}+1) \end{matrix} \right],\\
\end{aligned}
\]
where $t_j^+$ are as in Lemma \ref{lem:tj} and $a_i, b_i$ are non-zero constants.  From this it is apparent that the final formulas will be simpler if we conjugate $F$ by $T$, i.e. if we set $F = L\circ J$, where $L=T^{-1}\circ S$.  Letting $L_0,\dots,L_k$ denote the columns of (the matrix of) $L$, the above information about $S$ and $T$ tells us that  $L_j = \frac{b_j}{a_{j+1}} \ve_{j+1}$ for $0\leq j\leq k-1$.  

The fact that $\Lambda^{-1}\circ J = J\circ \Lambda$ for any invertible diagonal $\Lambda$ means that we can further conjugate by $\Lambda$ (or equivalently, replace $S$ and $T$ with $S\circ\Lambda$ and $T\circ\Lambda$)  in order to replace $L$ with $\Lambda^{-1} L \Lambda^{-1}$ to further simplify the columns $L_j$.  It seems convenient to us to make this choice so that $L(1,\dots,1) = (1,\dots,1)$ (i.e. both $T$ and $S$ send the fixed point $(1,\dots,1)$ of $J$ to the cusp $\ve_k = \gamma(\infty)$ of $C$).  This results in the following matrix for $L$
$$
\begin{pmatrix} 
0 & 0 &  \dots & 0 & 0 & 1 \\ 
b_0/a_1 & 0 & & 0 & 0 & 1-b_0/a_1 \\ 
0 & b_1/a_2 & & 0 & 0 & 1-b_1/a_2 \\ 
\vdots & & \ddots& & & \vdots \\ 
0 & 0 & &b_{k-2}/a_{k-1} & 0& 
1-b_{k-2}/a_{k-1}\\ 
0 & 0 & \dots & 0& b_{k-1}/a_k& 1-b_{k-1}/a_k\end{pmatrix}.
$$

We can then evaluate the entries below the main diagonal with the help of the following auxiliary result which we leave the reader to verify.

\begin{lem}
\label{linearsystem}
Suppose $M$ is a non-singular $k+1 \times k+1$ matrix whose $j$-th column is given by $\gamma(t_j)$. 
If a column vector $v= (v_0, \dots, v_k)^t$ satisfies $M. v = (0,\dots, 1)^t$ then for all $0 \le j \le k$ we have
\[ v_i =  \frac{1}{(\sum t_j)\prod_{j\ne i} ( t_j- t_i)}. \]
\end{lem}

The condition $T(1,\dots,1)=\ve_k$ amounts to setting $t_j = t_j^+$ and then taking $(a_1,\dots,a_k)$ to be the vector $v = (v_1,\dots,v_k)$ given in the conclusion of the lemma.    Hence with the help of Theorem \ref{good br maps}, we find that 
$$
a_i = \frac{k+1}{k-1}\cdot \frac{\delta}{1-\delta}\cdot\frac{1}{\Pi_{j\neq i} (t_j^+-t_i^+)}.
$$
Likewise, the condition $S(1,\dots,1)=\ve_k$ amounts to setting $t_j = t_{j+1}^+$ for $0\leq j\leq k-1$ and $t_k = (t_0^+-1)/\delta^{n-1}+1$ in Lemma \ref{linearsystem}, and then taking
$b_i = v_i$ as in the conclusion.  So applying Theorem \ref{good br maps} with $F^{-1}$ in place of $F$ gives
$$
b_{i-1} = \frac{k+1}{k-1}\cdot \frac{1}{\delta-1}\cdot\frac{1}{\Pi_{j\neq i}(t_j^+-t_i^+)}\cdot \frac{t_0^+-t_i^+}{(t_0^+-1)/\delta^{n-1}+1 - t_i^+)}.
$$
Therefore, the entries of $L$ below the main diagonal are
\begin{eqnarray*}
b_{i-1}/a_i & = & -\frac{1}{\delta}\cdot \frac{t_0^+-t_i^+}{(t_0^+-1)/\delta^{n-1}- (t_i^+-1)} 
= \frac{1}{\delta}\cdot
\frac{(\delta^{i+n-1}-\delta^{n-1})\frac{\delta^2-1}{\delta(\delta^{k+1}-1)}}
{(1-\delta^{n+i-1})\frac{\delta^2-1}{\delta(\delta^{k+1}-1)}+(\delta^{n-1}-1)} \\
& = &
\frac{1}{\delta}\cdot
\frac{(\delta^{i+n-1}-\delta^{n-1})}
{(1-\delta^{n+i-1})+(\delta^{n+k}-1)} 
=
\frac{\delta^i-1}{\delta(\delta^{k+1}-\delta^i)}.
\end{eqnarray*}
for $1\leq i\leq k$.  The second equality uses the formula for $t_i^+$ given in Lemma \ref{lem:tj}, and the third uses that $\delta$ is a root of the polynomial given in the same lemma.  In summary, we have just shown that the map $F:=L\circ J$ of Theorem \ref{better br map} has a very convenient expression in terms of the multiplier $\delta$:
%\footnote{In summary, we have just finished establishing that the map $F := L\circ J$ guaranteed by Theorem \ref{better br map} has a surprisingly tractable expression in terms of the multiplier $\delta$, proving}

\begin{thm}\label{T:L}
The matrix $L= T^{-1} S$ is given by
\[ L=  \begin{pmatrix} 0 & 0 &  & & 0 & 1 \\ \beta_1 & 0 & & & 0 & 1-\beta_1 \\ 0 & \beta_2 & 0 & & 0 & 1-\beta_2 \\ & & \ddots& \ddots& & \vdots \\ 0 & & 0&\beta_{k-1} & 0& 1-\beta_{k-1}\\ 0 & & & 0& \beta_k& 1-\beta_k\end{pmatrix}\]
and $\beta_i =  (\delta^i-1)/ (\delta (\delta^{k+1} - \delta^i))$ for $i=1, \dots, k$.
\end{thm}

\section{The connection with coxeter groups}\label{S:coxeter}
Let us consider a basic cremona map $F$ discussed in the \S 3.  Let $\rho: X \to \cp^k$ be the blowup of $N:=\sum n_i$ distinct points $\{ F^j(S(e_i)), 0 \le j \le n_i-1, 0\le i \le k\}$ as in Proposition \ref{pseudoaut}. Let $H$ denote the class of a generic hypersurface in $X$ and let $E_{i,j}$ denote the class of the exceptional divisor over $F^{j-1}(S(e_i))$:
\[ E_{i,j} = [ \rho^{-1} (F^{n_i-j}(S(e_i)))] \quad \text{for} \ \ 1 \le j \le n_i, \ \ 0 \le i \le k.\]
The Picard group of $X$ is given by
\[ 
\pic(X) = \langle H, E_{i,j} \ 1 \le j \le n_i, \ \ 0 \le i \le k\rangle. 
\]
Let us define a symmetric bilinear form $\langle \cdot, \cdot \rangle$ on $\pic(X)$ as follows:
\[ 
\langle \alpha,\beta \rangle = \alpha \cdot \beta\cdot \Phi \qquad \alpha, \beta \in \pic(X)
\]
where 
\[ 
\Phi= (k-1) H^{k-2} + (-1)^k \sum_{i,j} E_{i,j}^{k-2} \in H^{k-2,k-2}(X)
\]
and $D^n=D\cdot D \cdots D$ is a $n$-fold intersection product. Since $H^k=1, E_{i,j}^k= (-1)^{k-1}$ and everything else is zero, our choice of basis for $\pic(X)$ gives a geometric basis with respect to the bilinear form:
\[ 
\langle H, H\rangle = k-1, \ \ \ \langle E_{i,j}, E_{i,j}\rangle = -1 
\]
and zero otherwise.

\begin{rem}
In case $k=3$, $-2 \Phi= K_X$, where $K_X$ is the canonical class of $X$. 
\end{rem}

\begin{rem}
For $k \ge 3$, the anticanonical class $-K_X$ is given by 
\[ 
- K_X = (k+1) H - (k-1) \sum_{i,j} E_{i,j}. 
\]
It follows that the class of $-K_X \cdot \Phi$ in $H^{k-1,k-1}$ is the class of the invariant curve $[C]$. 
Thus for any hypersurface $D \subset X$, we have 
\[ 
\langle D, -K_X \rangle = D \cdot C
\]
is the number of intersections between $D$ and $C$, counted with multiplicity.
\end{rem}

\begin{rem}
Since there are $N$ blowups we have 
\[\langle K_X , K_X \rangle = (k+1)^2 (k-1) - (k-1)^2 N = - (k-1)^2 \left( N- \frac{(k+1)^2}{k-1} \right). \]
and since $N$ is a positive integer,  this equation  is not equal to zero unless $k-1$ divides $4$. In case $k=2,3,5$ one can check this equation vanishes if  $N=9,8,9$ respectively. 
Thus if $(k,N) \ne (2,9), (3,8)$, or  $(5,9)$ then $K_X \not\perp K_X$.
Notice that  in case $F: \cp^k \dasharrow \cp^k $ has the orbit data $(1,1,\dots, n)$ with the cyclic permutation, the total number of blowups is $N=k+n$. Thus if $(k,N) = (2,9), (3,8)$, or  $(5,9)$ then 
$(k,n) = (k, N-k)= (2,7), (3,5), (5,4) \in \Gamma$ where the set $\Gamma$ is defined in Section 4.  Thus for these three cases the first dynamical degree of $F$ is equal to $1$. 
\end{rem}

Observe that the $N$ dimensional subspace $K_X^\perp\subset\pic(X)$ may be decomposed into the one dimensional subspace generated by $\alpha_0 := H-\sum_{i,0}^k E_{i,1}$ and the complementary subspace generated by the elements $E_{i',j'}-E_{i,j}$.  Indeed, if we (re)label the exceptional curves $E_{i,j}$ as $E_0,\dots, E_{N-1}$, then the elements $\alpha_i := E_i-E_{i-1}$, $1\leq i\leq N-1$ give a basis for the latter subspace.  

In order to see the connection with Coxeter groups, we take in particular $E_i := E_{i,1}$, $0\leq i \leq k$, $E_{k+1} = E_{k,2}$.  The relabeling of the remaining $E_{i,j}$ may be chosen arbitrarily.  One then checks easily that $B = \{\alpha_0,\dots,\alpha_{N-1}\}$ is a basis for $K_X^\perp$ satisfying
\[
\begin{aligned}
& \langle \alpha_i, \alpha_i \rangle\  =\  -2\quad \text{for\ all\ } i=0, \dots, N-1, \\
& \langle \alpha_i, \alpha_{i+1} \rangle = 1\quad \text{for\ all\ } i=1, \dots, N-1, \\
& \langle \alpha_0, \alpha_{k+1} \rangle = 1, \ \ \langle \alpha_i,\alpha_j \rangle =0 \quad \text{otherwise}.
\end{aligned}
\]
That is, $( \langle \alpha_i, \alpha_j \rangle)_{0 \le i,j \le N-1}$ is $-2$ times the gram matrix of the Coxeter group of Type $T_{2, k+1, N-k-1}$. (See \ \cite{Humphreys, Mukai:2004}.) Let $s_i$ be the reflection through a hyperplane orthogonal to $\alpha_i$:
\[
s_i (D) = D+ \langle D, \alpha_i \rangle \, \alpha_i.
\]
The Coxeter group $W(2, k+1, N-k-1)$ is the group generated by such reflections.  Thus we have identified $W(2,k+1,N-k-1)$ with a subgroup of $\aut(\pic(X))$ that acts orthogonally relative to the inner product $\langle \cdot, \cdot \rangle$. 

\begin{thm}
The action $F_X^*$ of the pseudoautomorphism $F_X$ in Proposition \ref{pseudoaut} belongs to $W(2, k+1, N-k-1)$. Hence $F_X^*$ preserves the bilinear form $\langle \cdot, \cdot \rangle$. 
\end{thm}

\begin{proof}
The reflection $s_0$ corresponds to the action of $J^*$ on $\pic(X)$.  The remaining reflections $s_i$, $i\geq 1$ generate the group of permutations of the exceptional curves $E_{i,j}$ for the modification $X\to\cp^k$.  Therefore, this Theorem follows from the decomposition (which we leave the reader to check) $F_X^* = s_0\hat \sigma\pi_0\dots\pi_k$, where $\pi_i$ cyclicly permutes the curves $E_{i,j}$, $1\leq j\leq n_j$, and $\hat\sigma$ permutes the curves $E_{i,1}$, $0\leq i\leq k$ according to $\hat\sigma(E_{i,1}) = E_{\sigma(i),1}$, where $\sigma$ is the permutation in the hypothesis of Proposition \ref{pseudoaut}.
\end{proof}

\section{Pseudoautomorphisms on Multiprojective spaces}\label{S:multi}

The article \cite{perroni2011pseudo} considers the more general problem of existence of pseudoautomorphisms $F_X:X\dasharrow X$, where $\pi_X:X\to (\cp^k)^m$ is a modification of the multiprojective space $(\cp^k)^m$, obtained as before by blowing up distinct smooth points along an `elliptic normal curve' $C$.  Our method yields formulas for the pseudoautomorphisms that arise in this case, too.  Hence we conclude with a quick sketch of the computations that arise here, laying greatest emphasis on the way things differ from the work presented above.  For the sake of simplicity we specialize to the case $m=2$, i.e. $X$ is a modification of $\cp^k\times \cp^k$.  

Note first of all that $\pic(\cp^k\times\cp^k) \cong \Z^2$ is generated by `horizontal' and `vertical' hyperplanes $H := \cp^2\times L$ and $V := L\times \cp^k$, where $L\subset\cp^k$ is a generic hyperplane.  Let $\aut_0(\cp^k\times\cp^k)$ denote the connected component of the identity inside $\aut(\cp^k\times\cp^k)$.  The group $\aut_0(\cp^k\times \cp^k)$ consists of products $T_v\times T_h$ of linear maps $T_v,T_h\in\aut(\cp^k)$.  For every $a\in\C$, one has an embedding $\gamma_a:C\to X$ given by $t\mapsto (\gamma(t),\gamma(t-a))$, but in fact all embeddings of $C$ into $X$ that satisfy $C\cdot V = C\cdot H = k+1$ are equivalent via $\aut_0(\cp^k\times\cp^k)$ to either $\gamma_0$ or $\gamma_1$.  The latter, which we use here, is in some sense the generic case, distinguished from the former by the condition $(k+1)[\gamma_1(1)] = H|_C \neq V|_C = (k+1)[\gamma_1(0)]$. 

The `standard' cremona map $J$ is also different in this context.  If we write points in $(\cp^k)^2$ in bihomogeneous coordinates $(x,y) = ((x_0,\dots,x_k),(y_0,\dots,y_k))$, then $J$ is given by 
$$
J:(x,y) \mapsto (y/x,1/x) := ((y_0/x_0,\dots,y_k/x_k),(1/x_0,\dots, 1/x_k)).
$$
Note that $J$ contracts the $k+1$ vertical hyperplanes $x_j=0$ to diagonal points $(\ve_j,\ve_j)$.  Though not an involution, $J$ is reversible with $J^{-1}:(x,y)\mapsto (1/y,x/y)$ conjugate to $J$ via $(x,y)\mapsto (y,x)$.  Hence $J^{-1}$ contracts the horizontal hyperplanes $y_j = 0$ to the same diagonal points $(\ve_j,\ve_j)$.

We recycle the terminology from \S \ref{S:basic cremona maps}: a basic cremona transformation is one of the form $F:= S\circ J \circ T^{-1}$ with $S,T\in\aut(\cp^k\times\cp^k)$, $F$ is 
centered on $C$ if $T(\ve_j,\ve_j)\in C_{reg}$ for all $j$, and $F$ properly fixes $C$ if, in addition, $F(C) = C$.  Proposition \ref{centered} and Corollary \ref{centeredimage} apply to the present context with straightfoward modifications.  The analogue for Theorem~\ref{good br maps}, which we state next, differs in one important way from its predecessor.  While there exist automorphisms of $\cp^k\times\cp^k$ that preserve $C$, none of these restrict to $C_{reg} \cong \pic_0(C)$ as group automorphisms (in parametric terms, maps of the form $t\mapsto \delta t$).  Therefore, the multiplier $\delta$ for $F|_C$ must depend on the choice of parameters $t_j^+$ for the images of exceptional hyperplanes.

\begin{thm}
\label{good br maps 2}
Let $t^+_j\in\C$, $0\leq j\leq k$, be distinct parameters satisfying $\sum t_j^+\neq 0$.  Then there exists a unique basic cremona map $F=S\circ J\circ T^{-1}:\cp^k\times\cp^k\dasharrow\cp^k\times\cp^k$ properly fixing $C$ such that $\gamma_1(t_j^+) = T(\ve_j,\ve_j)$, $0\leq j\leq k$. 
The restriction $F|_C$ is given by $F\circ\gamma_1(t) = \gamma_1(\delta t + \tau)$, where $\delta\in\C^*$ and $\tau \in\C$ satisfy
\begin{itemize}
 \item $\sum t_j^+ = (k+1) (\delta^{-1}+1)$,
 \item $\tau = k + (k-1)\delta$, and 
 \item $S(\ve_j,\ve_j) = \gamma_1(\delta(t_j^+-2) - 1)$.
\end{itemize}
\end{thm}

We omit most of the proof, deriving only the formulas for $\tau$, $\delta$ and $S(\ve_j,\ve_j)$.  First note that regardless of $S$ and $T$, the induced action $F_*$ on 
$\pic(\cp^k\times\cp^k)$ is given by $F_*V = kH$, $F_*H = V+kH$.  Moreover, one computes directly that the images $F(V) = F_*V, F(H) = F_*H$ of generic vertical and horizontal hyperplanes contain the points $S(\ve_j,\ve_j)$ with multiplicity $k-1$ and $k$, respectively.  Hence assuming that $F$ properly fixes $C$, we infer that
$$
kH|_C = (F_*V)|_C = (F|_C)_*(V|_C) + \sum[(\ve_j,\ve_j)].
$$
In terms of parameters, we may write $S(\ve_j,\ve_j) = \gamma_1(t_j^-)$ and conclude that
$$
k(k+1) = (k+1)\tau +(k-1) \sum t_j^-.
$$
Similarly, considering the image of a horizontal hyperplane gives
$$
k(k+1) = (k+1)(\delta + \tau) + k\sum t_j^-.
$$
Together, the two equations imply $\tau = k + (k-1)\delta$ and $\sum t_j^- = -(k+1)\delta$.

Now consider e.g. the vertical hyperplanes $V_j := T(\{x_j=0\})$ contracted by $F$.  On the one hand $V_j|_C = [p_j] + \sum_{i\neq j} [T(\ve_i,\ve_i)]$ for some $p_j \in C$.  In terms of parameters, this becomes $s_j = - \sum_{i\neq j} t_i^+$, where $\gamma_1(s_j) = p_j$.  On the other hand $\gamma_1(\delta s_j + \tau) = F(p_j) = F(V_j) = S(\ve_j,\ve_j) = \gamma_1(t_j^-)$.  
Hence 
$$
t_j^- = \delta s_j + \tau = \tau -\delta\sum_{i\neq j} t_i^+,
$$ 
So summing the equation over all $j$ gives
$
\sum t_j^-  = (k+1)\tau - k\delta\sum t_j^+,  
$
which implies
$$
\delta\sum t_j^+ = (k+1) (\delta+1).
$$
Substituting this into the previous display, we arrive at
$$
t_j^- = \delta(t_j^+-2) - 1, 
$$
which is the parameter for $S(\ve_0,\ve_0)$.
\qed

As in \S \ref{S:Pk_coxeter}, one can choose explicit parameters $t_j^+$, $\delta$ in Theorem \ref{good br maps 2} so that the resulting map
$F=S\circ J\circ T^{-1}:\cp^k\times\cp^k\dasharrow\cp^k\times\cp^k$ has orbit data $(1,\dots, 1,n)$ with cyclic permutation (see \eqref{E:orbitdata}). Let us set \[ \Gamma_2=  \{ (2,3),(2,4), (3,3),(4,3),(5,3)\} \cup \{ ((k,1),(k,2),\  k\ge 2 \} \]

\begin{lem}
The multiplier $\delta$ is a root of the polynomial
%\[\chi_{k,n} = \delta ^n ( \delta ^{k+2} -\delta^k-2\sum_{j=1}^{k-1} \delta^j-1) + \delta^{k+2}+ 2 \sum_{j=3}^{k+1} \delta^j+ \delta^2 -1. \]
\[\chi_{k,n} = \delta ^n ( \delta ^{k+2} -\sum_{j=0}^{k} c_j \delta^j) + \delta^2 \sum_{j=0}^{k} c_{j} \delta^j -1 \]
where $c_0=c_{k} =1$, and  $c_1=c_2= \cdots = c_{k-1} =2$. Furthermore if $(k,n) \in \Gamma_2 $ then $\chi_{k,n}$ is a product of cyclotomic polynomials. If $(k,n) \not\in\Gamma_2$ then $\chi_{k,n}$ has a Salem polynomial factor and thus the largest real root is bigger than $1$.
\end{lem}

Starting from Theorem \ref{good br maps 2} in place of Theorem \ref{good br maps}, the proof of the first part of this lemma is essentially identical to the proof of Lemma \ref{lem:tj}. The polynomial $\chi_{k,n}$ is the characteristic polynomial of the generalized coxeter group $W(3, k+1, n-1)$, so it follows that $\chi_{k,n}$ is a product of cyclotomic polynomials and at most one Salem polynomial. As in Corollary~\ref{C:salem}, we see that the largest root of $\chi_{k,n}$ increases to a root of $\delta ^{k+2} -\delta^k-2\sum_{j=1}^{k-1} \delta^j-1$ as $n \to \infty$ and to a root of $\delta ^{n+2} -\delta^n-2\sum_{j=1}^{n-1} \delta^j-1$ as $k \to \infty$. For each $k\ge 2$, we have $\chi_{k,1}= (x^{k+1}-1) (x^2+x+1)$ and $\chi_{k,2}=x^{k+4}-1$. Checking the other five elements in $\Gamma_2$ directly we see that $\chi_{k,n}$ is a product of cyclotomic polynomial if $(k,n) \in \Gamma_2$. The largest root of $\chi_{2,5}$, $\chi_{3,4}$ , and $\chi_{6,3}$ are $1.40127$, $1.40127$, and $1.17628$ respectively. Thus, by monotonicity, we get the final assertion in the lemma. 

\begin{lem} 
The parameters $t_j^+$ are given by
 \[ t_j^+ \ =\ \frac{\delta^{j-1}}{\delta^{k+1}-1} \left[ k (k+1) - \delta (\delta+1) \right] - \frac{k-2 \delta^2-\delta+1}{\delta (\delta-1)} .\]
Furthermore all points $T(\ve_j,\ve_j)$, $1\leq j\leq k$ and $F^{-i}(T(\ve_0,\ve_0))$, $0\leq i \leq n-1$ are distinct. 
 \end{lem}
 
\begin{proof}
The given orbit data and permutation, together with Theorem \ref{good br maps}, give $t_i^-= \delta^i t_0^- + (k-2 \delta)(\delta^i-1)/(\delta-1)$ for $j=1, \dots, k$, and also $\sum {t_j^-} = -(k+1)\delta$ and $t_j^- = \delta (t_j^+-2)-1$.  The formula for $t_j^+$ follows from these equations.  

Now since \[t_j^+ - t_i^- = (\delta^j-\delta^i) (k (k+1) - \delta (\delta+1))/(\delta (\delta^{k+1}-1))\] it follows that $t_j^+ \ne t_i^+$ for $j \ne i$. Applying $F^{-i}$ to $T(\ve_0,\ve_0)$, we see that if there are $i$ and $j$ such that $T(\ve_j,\ve_j)=F^{-i}(T(\ve_0,\ve_0))$, then 
\[k-\delta^2-\delta - \delta (\delta^i-1) ( \delta^k + \delta^{k-1} + \cdots +1)=0. \]
 Since $\delta$ is a Galois conjugate of a Salem number, above equation should be divisible by a Salem polynomial. However if $\delta >1$ we see that the left hand side of the equation is strictly negative.
 \end{proof} 
 
Finally, we can imitate the argument for Theorem \ref{T:L} to get a formula for $T^{-1}S = L_1\times L_2\in \aut_0(\cp^k\times\cp^k)$. We let both $T$ and $S$ send the fixed point $((1,\dots, 1),(1,\dots, 1))$ of $J$ to the cusp of $C$ and thus $L$ fixes the point $((1,\dots, 1),(1,\dots, 1))$.
%??? This time we rely on the fixed point $p_f=((s_1,\dots, s_1), (s_2, \dots, s_2))$ ???? \mpar{Fixed point for which map?  What are $s_1$ and $s_2$?} 
%\vspace{2cm}

\begin{thm}
\label{multiprojL}
The matrix for $L_i\in\aut(\cp^k)$, $i=1,2$ is given by
\[ \begin{pmatrix} 0 & 0 &  & & 0 & s_i \\ \beta_1 & 0 & & & 0 & s_i-\beta_1 \\ 0 & \beta_2 & 0 & & 0 & s_i-\beta_2 \\ & & \ddots& \ddots& & \vdots \\ 0 & & 0&\beta_{k-1} & 0& s_i-\beta_{k-1}\\ 0 & & & 0& \beta_k& s_i-\beta_k\end{pmatrix}\]
where $s_1=1, s_2= (\delta^2+\delta+1)/\delta$ and $\beta_j = (\delta^j-1)(\delta+1) /(\delta^2 (\delta^{k+1} - \delta^j))$ for $j=1, \dots, k$.
\end{thm}

{\bf Concluding remarks.}
So far, we have described a construction of pseudo automorphisms which is achieved by blowing up points on the elliptic normal curve.  The same procedure works with other invariant curves.  Two of these that occur in all dimensions are: (i) the rational normal curve and a tangent line, (ii) $k+1$ concurrent lines in general position.  We will make a few comments on case (ii).  First we work in $\cp^k$, and then we consider multi-projective space.

In the case of concurrent lines, we let $L_j$, $0\le j\le k$, denote the line passing through $[1:\cdots:1]$  and $e_j$.   For the parametrizations, we may use:  $\psi_c^{(j)}:{\bf C}\to{\bf P}^k$ with $\psi_c^{(0)}(t) = [-t:1:\cdots:1]$, and for $1\le j\le k$, $\psi_c^{(j)}(t) =[t:0:\cdots:1:\cdots:0]$, where there is one `1', and this appears in the $j$th slot.    If we wish to work on multi-projective spaces, we use the parametrized curve $\Psi:{\bf C}\to( {\bf P}^k)^m={\bf P}^k\times\cdots\times{\bf P}^k$ given by $\Psi(t) = (\psi(t-\tau_0),\psi(t-\tau_1),\dots,\psi(t-\tau_{m-1}))$.  

Now let us consider multi-projective spaces $({\bf P}^k)^m={\bf P}^k\times\cdots\times{\bf P}^k$.   We write a point as $(x,y)=(x,y^{(1)},\dots,y^{(m-1)})\in (\cp^k)^m$.  As a basic Cremona map, we start with 
$$J(x,y):(x,y^{(1)},\dots,y^{(m-1)})\mapsto (y^{(1)}/x,\dots,y^{(m-1)}/x,1/x),$$ 
where as before $y^{(s)}/x= [y^{(s)}_0/x_0:\cdots:y^{(s)}_k/x_k]$.  The exceptional hypersurfaces are given, as in the case $m=2$, by $J: \{x_j=0\}\mapsto (\ve_j,\dots,\ve_j)$.  

%In this case the dynamical degree $\delta(f)=\delta_{n,k,m}$ is determined by the numbers $n$, $k$, and $m$, where $n$ is the orbit length that is greater than 1.  The dynamical degree is a root of a characteristic polynomial $\chi$, which can be computed by techniques that are by now standard.  Carrying this out yields the following polynomial  
%$$\chi(t)=t^n\eta(t)+t^{k+m}\eta\left(\frac {1}{t}\right), \text{ where}\ \  \eta(t) = t^{k+m}-\sum_{j=0}^{k+m-2} c_jt^j$$
%where to define the coefficients we set $\nu=\min\{k,m\}$ and then  $c_j=c_{m+k-j-2}=j+1$ for $0\le j\le \nu-2$ and $c_j=\nu$ for $\nu-1\le j\le m+k-\nu-1$.  We may factor $\chi=\chi_c \chi_s$, where $\chi_c$ is a product of cyclotomic factors, and $\chi_s$ is a Salem polynomial.

% Let $\psi$ denote any of the sets of parametrized curves listed above: elliptic normal, rational normal, or concurrent lines.  If we wish to work on multi-projective spaces, we use the parametrized curve $\Psi:{\bf C}\to( {\bf P}^k)^m={\bf P}^k\times\cdots\times{\bf P}^k$ given by $\Psi(t) = (\psi(t-\tau_0),\psi(t-\tau_1),\dots,\psi(t-\tau_{m-1}))$.  Now let  $\alpha$ be a Galois conjugate of the dynamical degree $\delta_{n,k,m}$, which is to say, any root of $\chi_s$.  We may set $\tau_0=0$, and $\tau_j=1+\alpha+\cdots+\alpha^{j-1}$.

With the curve $\Psi$, it is possible to carry through the same principle of construction as in the preceding sections.  We consider the case (ii) of concurrent lines and give the map $L=L_0\times\cdots\times L_{m-1}\in\aut_0(\cp^k\times\cdots\times\cp^k)$ so that the map $f:=L\circ J$ will have orbit data $\{(1,\dots,1,n(k+1)),\sigma\}$, where $\sigma$ is a cyclic permutation.  The orbit length is divisible by $k+1$ because the orbit of $\Sigma_k$ moves cyclically  through each of the $k+1$ lines.   With this orbit data, the resulting pseudo-automorphism will represent the Coxeter element of a $T$-shaped diagram  \cite{Mukai:2004}.   We let $\alpha$ be any Galois conjugate of the dynamical degree $\delta$ for this orbit data, and  the desired  matrices are given by 
$$L_j =\left(\begin{array} {ccccc}
0 & 0 & 0 & 0 & s_j \\ 
v&  0&0&0& s_j-v\\
0&v&0&0 &s_j-v\\
0&0&\ddots &0&s_j-v\\
0&0&0&v&s_j-v \end{array}\right), v = -\alpha\frac{\alpha^m-1}{\alpha-1},   \  s_j = \frac{(\alpha^m-1)(\alpha^{j+1}-1)}{\alpha^{j}(\alpha-1)(\alpha^{m-j}-1)}$$
for $j=0,\dots,m-1$.

%\vspace{2cm}

%References to add:  Dolgachev (BAMS survey), Dolgachev-Ortland, McMullen (Blowups of ${\bf P}^2$)

\bibliographystyle{plain}
\bibliography{bibliopseudo}
\nocite{}

\end{document}

%% file: shortcuts.tex
\newtheorem{thm}{Theorem}[section]
\newtheorem{lem}[thm]{Lemma}

\newtheorem{cor}[thm]{Corollary}
\newtheorem{rem}[thm]{Remark}

\newtheorem{prop}[thm]{Proposition}

\newcounter{probno}
\stepcounter{probno}

\renewcommand{\qed}{\hfill$\Box$ \par \medskip}

\newcommand{\self}{\circlearrowleft}

\newcommand{\id}{\mathrm{id}}

\newcommand{\C}{\mathbf{C}}
\newcommand{\cp}{\mathbf{P}}
\newcommand{\Z}{\mathbf{Z}}

\newcommand{\supp}{\mathrm{supp}\,}

\newcommand{\pic}{\mathrm{Pic}\,}